\documentclass[12pt,a4paper]{article}
\usepackage{graphicx,psfrag,amsmath,amsfonts,verbatim}
\usepackage{amssymb}
\usepackage{mathrsfs}
\usepackage{amsthm}
\usepackage{enumerate}
\usepackage{geometry}
\usepackage{dsfont}
\usepackage{bm}
\usepackage{cite}
\usepackage{appendix}
\usepackage{url}
\usepackage{color}
\usepackage{multirow}
\usepackage{threeparttable}
\usepackage{booktabs}
\usepackage{array}
\usepackage{subcaption}
\usepackage{accents}
\usepackage{statex}
\usepackage[colorlinks=true,linkcolor=black,anchorcolor=black,citecolor=black,urlcolor=blue]{hyperref}
\makeatletter

\def\wideubar{\underaccent{{\cc@style\underline{\mskip8mu}}}}

\newtheorem{lemma}{Lemma}[section]
\newtheorem{theorem}{Theorem}[section]
\newtheorem{assumption}{Assumption}[section]

\title{Globalized distributionally robust optimization with multi core sets}
\author{Yueyao Li, Chenglong Bao, Wenxun Xing}
\date{}

\begin{document}

\maketitle

\begin{abstract}
It is essential to capture the true probability distribution of uncertain data in the distributionally robust optimization (DRO).  
The uncertain data presents multimodality in numerous application scenarios, in the sense that the probability density function of the uncertain data has two or more modes (local maximums). 
In this paper, we propose a globalized distributionally robust optimization framework with multiple core sets (MGDRO) to handle the multimodal data. This framework captures the multimodal structure via a penalty function composed of the minimum distances from the random vector to all core sets. 
Under some assumptions, the MGDRO model can be reformulated as tractable semi-definite programs for both moment-based and metric-based ambiguity sets. 
We applied the MGDRO models to a multi-product newswendor problem with multimodal demands. The numerical results turn out that the MGDRO models outperform traditional DRO models and other multimodal models greatly.
\end{abstract}

\section{Introduction}\label{sec:intro}
The distributionally robust optimization (DRO) method allows decision makers to handle the uncertain parameters in the optimization problems, via constructing an ambiguity set to characterize the probability distribution of the uncertain parameters. In this paper, we investigate a complicated case when the uncertain parameters are multimodal, and assume that there is only a set of samples of the uncertain parameters available. It is an underlying challenge to capture the multimodality of the uncertain data in DRO when modelling with only a set of samples.

Significantly, multimodality of uncertain data is widely observed in the real life, such as demands in newsvendor problems \cite{dro:multimodal}, renewable data streams and forecast errors in energy management \cite{dro:ehs,dro:hmmg,wind}, and many engineering variables \cite{reliability}. In this paper, the multimodality of the probability distribution means that the probability density function of each random variable has two or more modes (local maximums), and the multimodality of uncertain data means that the data follows a multimodal probability distribution. The multimodal distribution may also be referred to as the mixture distribution in other papers \cite{drcc:mix,portfolio:mix,rsome}, since a multimodal distribution may be obtained from a mixture of some unimodal distributions. Thus the mixture components of the multimodal distribution can be regarded as the distribution of the random variables under different states.

The DRO method usually characterizes the probability distribution by the ambiguity set. 
The moment-based ambiguity set is widely used in DRO models due to tractable computability, which mainly describes the moments information (such as the mean and variance) of the probability distribution, see for example \cite{dro:ye,dro:matrix,dro:linear,dro:convex,dro:approx}. However, the moments information, especially the frequently-used lower order moments information, is hardly able to capture the multimodality of the data. Therefore, it may cause over-conservatism when the uncertain data is multimodal.

The metric-based ambiguity set generally characterizes the distance between the probability distribution and a nominal distribution, where $\phi$-divergence distance \cite{dro:phi} and Wasserstein distance \cite{dro:wasserstein,drccp:wasserstein} are commonly used to measure the distance between two probability distributions, and the nominal distribution is usually taken as the empirical distribution.  
The metric-based ambiguity set with the empirical distribution as nominal distribution captures the multimodality spontaneously, because the empirical distribution reflects the multimodality implied by the samples. However, the effect of the metric-based ambiguity set capturing the multimodality also depends on the radius. When the radius of the metric-based ambiguity set is large, it may lead to over-conservatism as well.

In a word, in terms of uncertainty modelling in DRO, it always leads to over-conservatism if the multimodality of the uncertain data is ignored. 
Therefore, we propose a globalized distributionally robust optimization framework with multiple core sets. It efficiently handles the multimodality of uncertain data, and improves the conservatism of DRO with moment-based ambiguity set significantly, even improves the conservatism of DRO with metric-based ambiguity set as well.
In this framework, the globalized distributionally robust optimization (GDRO) model, introduced by our previous work \cite{gdro}, is extended to the case with multiple core sets, accordingly denoted as MGDRO model.

In the MGDRO model, the sample space is large enough to contain the support set of the true probability distribution so that the robustness is guaranteed. Multiple core sets are constructed to capture all clustering regions of the sample points, from the perspective of the center, shape and scale. Similarly to that in GDRO \cite{gdro}, the core sets mainly contain the random vectors that are more likely to occur under each state. 
Then the degree of conservatism can be controlled by penalizing the minimum distance between the random vector and multiple core sets in the objective. Different penalty coefficients and distance functions can also be adopted to assist in capturing different probability densities in each region. The penalty item weakens the impact of the regions outside of the core sets on the expectation of the objective function, so that the multimodality is well characterized and the degree of conservatism is reduced.

We are aware of that the term ``coreset" has been widely used as a data reduction/summarization method, for turning big data into tiny data \cite{coreset,coreset2}. Such coreset is a compressed data set, which takes up significantly less space in memory, but approximates the original large data set, as stated in \cite{coreset}. 
It has been applied to DRO problems with Wasserstein-distance-based ambiguity set for reducing the computational complexity \cite{dro:coreset}. Different from that, the core sets we adopted here are some small sets containing the random vectors with high probability, which aim to capture the multimodality of the uncertain data and reduce the degree of conservatism, and take no account of the sample size. 

Especially, there are a few works constructing structual ambiguity set to capture the multimodality of random vectors in DRO.
Hanasusanto et al.\cite{dro:multimodal} constructed a mixture of some moment-based ambiguity sets with different supports for distributionally robust multi-item newsvendor problems. The different supports represent the supports of data under different states, and the moments in each ambiguity set is the conditional moments under every state.
Zhao et al.\cite{dro:ehs} proposed a two-stage DRO model with multimodal ambiguity set for energy hub systems. Their multimodal ambiguity set, based on the work of Hanasusanto et al.\cite{dro:multimodal}, added an uncertain variable $\tilde{s}$ to represent modality information and also adopted the equalities involved with the conditional moments with respect to $\tilde{s}$ as constraints.
Besides, Moulton \cite{rf} proposed a robust fragmentation model, fragmenting the support set of the uncertain data into some subregions $\{\Omega_k\}_{k=1}^K$ with probabilities $\{p_k\}_{k=1}^K$. They used equality constraints on conditional moments under each subregion and an inequality constraint on $\bm{p}$ to construct the ambiguity set. 
The event-wise ambiguity set proposed by Chen et al.\cite{rsome}, can also be adopted to model ambiguous multimodal distributions.

The works mentioned above and our MGDRO framework all consider utilizing some subregions of the support set to handle the multimodality of the uncertain data. However, there are some differences in the choices and meanings of the subregions. In their works, the subregions are assumed to be the support sets under each state, or the union of all subregions is assumed to contain the whole support set of the true probability distribution. In the MGDRO framework, a sufficiently large sample space contains the support set of the true probability distribution, while the union of the core sets does not necessarily contain the support set. The core set for every state only contains the random vectors that are more likely to occur, instead of all sample points belonging to that state. In addition, we handle the multimodality of the uncertain data by a penalty function in the optimization objective, while they design a structual ambiguity set involved with some conditional moment constraints. In the MGDRO model, the ambiguity set can be chosen as a moment-based or metric-based one, capturing some original information of the sample data. For instance, in the moment-based ambiguity set, the moments on the full support set instead of conditional moments are adopted, since we hope to use original moments information of the sample data and avoid extra errors caused by a clustering algorithm as much as possible.

We highlight the following main contributions of this paper.
\begin{itemize}
\item We propose a globalized distributionally robust optimization framework with multiple core sets, as an extension of the GDRO \cite{gdro}, efficiently handling the multimodality of the uncertain data. 
\item We prove that both of the MGDRO models with the moment-based and metric-based ambiguity set can be reformulated as semi-definite programs (SDP), with some appropriate choices of the sample space, core sets and distance function, under the assumption that the objective function is piecewise linear.
\item The MGDRO framework can be successfully applied to a multi-product newsvendor problem with multimodal demands and shows great advantages compared with some other models.
\end{itemize}

The remainder of this paper is organized as follows. In Section \ref{sec:gdro}, we introduce the MGDRO model and derive the computationally tractable reformulation under certain assumptions. Section \ref{sec:app} presents the application on a multi-product newsvendor problem and some numerical results. Section \ref{sec:cons} presents the conclusions.

\emph{Notations.}  $\bm{x} \in \mathbb{R}^n$ is the decision vector, and $X \subseteq \mathbb{R}^n$ is the feasible region of $\bm{x}$. $\Xi \subseteq \mathbb{R}^p$ is the sample space, and $\mathcal{B}(\Xi)$ is the Borel $\sigma$-algebra on $\Xi$. $\mathcal{P}(\Xi)$ is the set of all probability measures or distributions on measurable space $(\Xi,\mathcal{B}(\Xi))$. $\mathcal{M}^d(\Xi) \subseteq \mathcal{P}(\Xi)$ is the set of probability distributions with $E_P[||\bm{\xi}||^d]<\infty$. $\bm{\xi}$ is a random vector that satisfies $\bm{\xi}(\bm{\omega})=\bm{\omega}$ in probability space $(\Xi,\mathcal{B}(\Xi),P)$. $E_P[\cdot]$ represents the expectation with respect to $\bm{\xi}$. The set of samples is denoted as $\{\hat{\bm{\xi}}_j\}_{j=1}^N$, and the empirical distribution $\hat{P}_N= \frac{1}{N} \sum_{j=1}^N \delta_{\hat{\bm{\xi}}_j}$ is the uniform distribution on $\{\hat{\bm{\xi}}_j\}_{j=1}^N$, where $\delta_{\hat{\bm{\xi}}_j}$ is the Dirac point mass at $\hat{\bm{\xi}}_j$.
$\bm{\mu}_0$ is the sample mean, and $\bm{\Sigma}_0$ is the sample covariance matrix. $Y_i,i=1,\ldots,m$ are the core sets of MGDRO models. $\mathbb{R}^n$ is the set of real $n$-dimensional vectors; $\mathbb{S}^n$ is the set of real $n\times n$ symmetric matrices; and $\mathbb{S}^n_+$ and $\mathbb{S}^n_{++}$ are the sets of positive semi-definite and positive definite matrices in $\mathbb{S}^n$, respectively. $\bm{e}_p \in \mathbb{R}^p$ is a vector of 1, and $\bm{I}_p \in \mathbb{R}^{p \times p}$ is the identity matrix. For any function $g:\mathbb{R}^n \to \mathbb{R}$, we let $\mathrm{dom}(g)=\{\bm{x} \in \mathbb{R}^n | g(\bm{x}) < \infty \}$. The convex conjugate function of $g$ is defined as
$$g^*(\bm{y})=\sup \limits_{\bm{x}\in \mathrm{dom}(g)} \{\bm{y}^{T} \bm{x} -g(\bm{x})\}.$$
The concave conjugate function of $g$ is defined as
$$g_*(\bm{y})=\inf \limits_{\bm{x}\in \mathrm{dom}(-g)} \{\bm{y}^{T} \bm{x} -g(\bm{x})\}.$$
For a function $f(\cdot,\cdot)$ with two variables, let $f^*(\cdot,\cdot)$ and $f_*(\cdot,\cdot)$ denote the convex and concave conjugate functions with respect to the first variable, respectively, and let $f^{**}(\cdot,\cdot)$ and $f_{**}(\cdot,\cdot)$ denote the convex and concave conjugate functions with respect to both variables, respectively. The indicator function for set $S$ is defined as follows:
\begin{equation*}
   \delta(\bm{x}|S)=
   \left\{ \begin{aligned}  0 \qquad &\mathrm{if} \quad \bm{x} \in S, \\  \infty \qquad & \mathrm{otherwise}. \end{aligned}\right.
\end{equation*}

\section{MGDRO models and reformulations}
\label{sec:gdro}
\subsection{Introduction to MGDRO models}
The basic DRO model is generally formulated as
\begin{equation}
\label{eq:dro}
\min \limits_{\bm{x} \in X} \max \limits_{P \in \mathcal{D}} E_{P}[h(\bm{x},\bm{\xi})],
\end{equation}
where $\bm{x} \in \mathbb{R}^n$ is the decision vector, $X \subseteq \mathbb{R}^n$ is a closed convex set of feasible solutions, $h(\bm{x},\bm{\xi})$ is a convex function in $\bm{x}$ that depends on the parameter $\bm{\xi} \in \mathbb{R}^p$, $\mathcal{D} \subseteq \mathcal{P}(\Xi)$ is the ambiguity set, $\mathcal{P}(\Xi)$ is the set of all probability measures on the measurable space $(\Xi,\mathcal{B}(\Xi))$, $\Xi \subseteq \mathbb{R}^p$ is the sample space, and $\mathcal{B}(\Xi)$ is the Borel $\sigma$-algebra on $\Xi$.

To capture the multimodality of the data, we present an extension of the GDRO \cite{gdro} model, denoted as MGDRO, which realizes the generalization from a single core set to multiple core sets. 
\begin{equation}
\label{eq:mgdro1}
\min \limits_{\bm{x} \in X} \max \limits_{P \in \mathcal{D}} E_{P}[h(\bm{x},\bm{\xi})- \min \limits_{1 \leq i \leq m} \theta_i \min \limits_{\bm{\xi}' \in Y_i} \phi_i(\bm{\xi},\bm{\xi}')],
\end{equation}
where the core sets $Y_i \subseteq \Xi, i=1,\ldots,m$ are convex and compact, $\theta_i > 0, i=1,\ldots,m$ are the penalty coefficients, $\phi_i:\mathbb{R}^p \times \mathbb{R}^p \to \mathbb{R}, i=1,\ldots,m$ are closed, jointly convex and nonnegative functions for which $\phi_i(\bm{\xi},\bm{\xi})=0$ for all $\bm{\xi} \in \mathbb{R}^p$. The MGDRO model is reduced to the GDRO model \cite{gdro} when $m=1$. 

The MGDRO model (\ref{eq:mgdro1}) can be viewed as a penalized framework, where the penalty item is the minimum distance from the random vector to $m$ core sets. The penalty item weakens the impact of the regions outside of the core sets on the expectation of the objective function, thereby highlighting the impact of the multimodality. 
The MGDRO model (\ref{eq:mgdro1}) is reduced to the DRO model (\ref{eq:dro}) if there exists a certain $\theta_{i_0} = 0$. Therefore, we require that $\theta_i > 0, i=1,\ldots,m$. 
The larger $\theta_{i_0}$ means the less attention paid to the region outside of $Y_{i_0}$. 
If all $\theta_{i} \to \infty, i=1,\ldots,m$, the MGDRO model (\ref{eq:mgdro1}) becomes the DRO model (\ref{eq:dro}) with $\Xi = \cup_{i=1}^m Y_{i}$.

In this paper, the core sets and sample space are defined by
\begin{equation}\label{eq:mcsandss}
\begin{split}
Y_i & = \{ \bm{\xi}_i + \bm{A}_i \bm{\zeta} | \bm{\zeta} \in Z_{1i} \}, i=1,\ldots,m \\
\Xi & = \{ \bm{\xi}_0 + \bm{A} \bm{\zeta} | \bm{\zeta} \in Z_2 \},
\end{split}
\end{equation}
where $\bm{\xi}_0$ is the ``nominal value" of $\bm{\xi}$, $\bm{\xi}_i$ is the center of the region $i$ that can be found by clustering algorithms, $\bm{A}_i$ and $\bm{A}$ are the ``perturbation matrices", $\bm{\zeta} \in \mathbb{R}^L$ is the random vector of ``primitive uncertainties", $Z_{1i}$ is a nonempty, compact and convex set with $0 \in ri(Z_{1i})$, $Z_2$ is also convex, and $\bm{\xi}_0,\bm{\xi}_i,\bm{A}_i,\bm{A},Z_{1i},Z_2$ satisfy $Y_i \subseteq \Xi$.

The moment-based and metric-based ambiguity sets, denoted as $\mathcal{D}_M$ and $\mathcal{D}_W$ respectively, are both considered in this paper.
The moment-based ambiguity set considered here is the same as that in \cite{dro:ye}.
\begin{equation}\label{eq:ambiguity1}
   \mathcal{D}_M=
   \left\{
    P \in \mathcal{P}(\Xi)
   \middle\arrowvert
   \begin{array}{lcl}
    & (E_{P}[\bm{\xi}]-\bm{\mu}_0)^{T}\bm{\Sigma}_0^{-1}(E_{P}[\bm{\xi}]-\bm{\mu}_0) \leq \gamma_1  \\
    & E_{P}[(\bm{\xi}-\bm{\mu}_0)(\bm{\xi}-\bm{\mu}_0)^{T}] \preceq \gamma_2\bm{\Sigma}_0\\
    & P(\bm{\xi} \in \Xi) = 1
   \end{array}
   \right\},
\end{equation}
where $\bm{\mu}_0$ is the sample mean, $\bm{\Sigma}_0$ is the sample covariance matrix, $\gamma_1$ and $\gamma_2$ are estimated by statistical methods to guarantee that the true mean and covariance matrix are contained in confidence regions with high probability. The specific estimator of $\gamma_1$ and $\gamma_2$ can refer to \cite{dro:ye,gdro}.

The metric-based ambiguity set considered here is the same as that in \cite{dro:wasserstein}, which uses Wasserstein distance to measure the distance between two probability distributions.
\begin{equation}\label{eq:ambiguity_w}
\mathcal{D}_W = \{P \in \mathcal{M}^d(\Xi) \arrowvert W_d(P,\hat{P}_N) \leq r \},
\end{equation}
where $\hat{P}_N = \frac{1}{N} \sum_{j=1}^N \delta_{\hat{\bm{\xi}}_j}$ is the empirical distribution, $\delta_{\hat{\bm{\xi}}_j}$ is the Dirac point mass at the $i$-th sample $\hat{\bm{\xi}}_j$, $\mathcal{M}^d(\Xi)$ is the set of probability distributions with $E_P[||\bm{\xi}||^d]<\infty$,
and the Wasserstein distance is defined as 
\begin{equation}\label{eq:wass}
\footnotesize
W_d(P,Q) = \inf \left\{\left(\int_{\Xi^2} ||\bm{\xi} - \bm{\eta}||^d \Pi(d\bm{\xi},d\bm{\eta})\right)^{1/d} \middle\arrowvert \begin{split} \Pi \text{ is a joint distribution of} \\ \bm{\xi} \text{ and } \bm{\eta} \text{ with marginals}\\ P \text{ and } Q, \text{ respectively} \end{split} \right\}.
\end{equation}
In the following, we consider the simple case $d=1$ and omit $d$ accordingly.

\subsection{Reformulations of MGDRO models}\label{sec:refom}
The MGDRO model, as an extension of the GDRO model, can also be reformulated similarly as that in \cite{gdro}. Therefore, the same assumption with respect to the objective function is made as follows.

\begin{assumption}\label{as:plinear}
$h(\bm{x},\bm{\xi})$ is a piecewise linear function in $\bm{\xi}$ as follows:
$$h(\bm{x},\bm{\xi})=\max \limits_{k\in \{1,2,...,K\}} \{\bm{a}_k(\bm{x})^{T} \bm{\xi}+b_k(\bm{x})\},$$
where $\bm{a}_k(\bm{x})$ and $b_k(\bm{x})$ are linear in $\bm{x}$ for all $k$.
\end{assumption}

Here we introduce a key lemma for reformulating the MGDRO models.

\begin{lemma}[Ben-Tal et al. \cite{gro} Theorem 1]
\label{lemma:gro}
Let $f(\cdot,\bm{x})$ be a closed proper concave function in $\mathbb{R}^p$ for all $\bm{x} \in \mathbb{R}^n$ and $\phi:\mathbb{R}^p \times \mathbb{R}^p \to \mathbb{R}$ be a closed, jointly convex, and nonnegative function for which $\phi(\bm{\xi},\bm{\xi})=0$ for all $\bm{\xi} \in \mathbb{R}^p$. Let  the set $U_1 \subseteq \mathbb{R}^p$ be nonempty, convex and compact with $0 \in ri(U_1)$, the set $U_2$ be closed and convex, satisfying $U_1 \subseteq U_2$.

Then $\bm{x}$ satisties:
\begin{equation}
\label{eq:grocons1}
f(\bm{\xi},\bm{x}) \leq \min \limits_{\bm{\xi}' \in U_1} \phi(\bm{\xi},\bm{\xi}') \qquad \forall \bm{\xi} \in U_2,
\end{equation}
if and only if there exist $\bm{v},\bm{w} \in \mathbb{R}^p$ satisfying the single inequality
\begin{equation}
\label{eq:grocons2}
\delta^*(\bm{v}|U_1) + \delta^*(\bm{w}|U_2)-f_*(\bm{v}+\bm{w},\bm{x})+\phi^{**}(\bm{v},-\bm{v}) \leq 0.
\end{equation}
\end{lemma}

Lemma \ref{lemma:gro} transforms the semi-infinite constraint (\ref{eq:grocons1}) into a single inequality (\ref{eq:grocons2}), and decouples the items involved with $U_1$, $U_2$ and $\phi$. Thus, it is essential to the computationally tractable reformulation of the MGDRO model.

\subsubsection{MGDRO model with moment-based ambiguity set}
Since the same ambiguity set with \cite{dro:ye} is considered here, the following lemma contributes to the reformulation of the MGDRO model with the moment-based ambiguity set $\mathcal{D}_{M}$.
\begin{lemma}[Delage and Ye \cite{dro:ye} Lemma 1]\label{lemma:dual}
Suppose  $\gamma_1 >0,\gamma_2 >0$ and $\bm{\Sigma}_0\in \mathbb{S}^p_{++}$ in the ambiguity set (\ref{eq:ambiguity1}), and that $h(\bm{x},\bm{\xi})$ is integrable for all $P \in \mathcal{D}$ for fixed $\bm{x} \in X$. Then DRO (\ref{eq:dro}) with ambiguity set (\ref{eq:ambiguity1}) can be reformulated as 
\begin{equation}\label{eq:redro}
\begin{split}
\min \limits_{\bm{x},\bm{\Lambda},\bm{q},t} \quad & t+\bm{\Lambda} \cdot (\gamma_2 \bm{\Sigma}_0 +\bm{\mu}_0 \bm{\mu}_0^{T})+ \sqrt{\gamma_1}||\bm{\Sigma}_0^{1/2} (\bm{q}+2\bm{\Lambda} \bm{\mu}_0)||_2 + \bm{q}^{T} \bm{\mu}_0 \\
s.t. \quad & h(\bm{x},\bm{\xi})-\bm{\xi}^{T}\bm{\Lambda} \bm{\xi} - \bm{q}^{T} \bm{\xi} \leq t \qquad \forall \bm{\xi} \in \Xi, \\
& \bm{\Lambda} \in \mathbb{S}_+^p, \quad \bm{q} \in \mathbb{R}^p, \quad t\in\mathbb{R}, \\
& \bm{x} \in X\subseteq \mathbb{R}^n,
\end{split}
\end{equation}
where $\bm{\Sigma}_0^{1/2}$ is the decomposition of $\bm{\Sigma}_0$ such that: $\bm{\Sigma}_0=(\bm{\Sigma}_0^{1/2})^T\bm{\Sigma}_0^{1/2}$.
\end{lemma}

Lemma \ref{lemma:dual} effectively deals with the $\min_x$ and $\max_P$ operations in (\ref{eq:dro}) by Lagrangian dual of the inner maximation problem, and removes the difficulty of calculating the expectation.

Based on these preparations, the computationally tractable reformulation of MGDRO model with ambiguity set $\mathcal{D}_{M}$ is presented in Theorem \ref{thm:mgdro1}.
\begin{theorem}\label{thm:mgdro1}
Suppose $\gamma_1 >0,\gamma_2 >0,\bm{\Sigma}_0\in \mathbb{S}^p_{++}$ and Assumption \ref{as:plinear} are satisfied. For a fixed $\bm{x} \in X$, we assume that $h(\bm{x},\bm{\xi})- \min_{1 \leq i \leq m} \theta_i \min_{\bm{\xi}' \in Y_i} \phi_i(\bm{\xi},\bm{\xi}')$ is integrable for all $P \in \mathcal{D}$. $Y_i$ and $\Xi $ are defined as (\ref{eq:mcsandss}). Then, MGDRO model (\ref{eq:mgdro1}) with ambiguity set (\ref{eq:ambiguity1}) can be reformulated as 
\begin{small}
\begin{equation}\label{eq:mgdrc1}
\begin{split}
\min \limits_{\mbox{\scriptsize$\begin{array}{c}\bm{x},\bm{\Lambda},\bm{q},t,\\ \bm{v}_{ik}, \bm{w}_{ik}, z^1_{ik}, z^2_{ik} \end{array}$}} \quad & t+\bm{\Lambda} \cdot (\gamma_2 \bm{\Sigma}_0 +\bm{\mu}_0 \bm{\mu}_0^{T})+ \sqrt{\gamma_1}||\bm{\Sigma}_0^{1/2} (\bm{q}+2\bm{\Lambda} \bm{\mu}_0)||_2 + \bm{q}^{T} \bm{\mu}_0 \\
s.t. \quad
&\begin{bmatrix}
      \bm{\Lambda} & \dfrac{1}{2} (\bm{v}_{ik} + \bm{w}_{ik} +\bm{q}-\bm{a}_k(\bm{x}))\\[1.5em]
      \dfrac{1}{2}(\bm{v}_{ik} + \bm{w}_{ik} +\bm{q}-\bm{a}_k(\bm{x}))^{T} & \begin{aligned}t-b_k(\bm{x})-\bm{\xi}_i^{T}\bm{v}_{ik} - \bm{\xi}_0^{T}\bm{w}_{ik}\\- \theta_i \phi_i^{**}(\frac{\bm{v}_{ik}}{\theta_i},-\frac{\bm{v}_{ik}}{\theta_i})-z^1_{ik} -z^2_{ik} \end{aligned}\\
      \end{bmatrix}  \succeq 0, \\[0.5em]
& \delta^*(\bm{A}_i^{T}\bm{v}_{ik}|Z_{1i}) \leq z^1_{ik}, \\[0.5em]
& \delta^*(\bm{A}^{T}\bm{w}_{ik}|Z_{2}) \leq z^2_{ik}, \\[0.5em]
& \bm{v}_{ik},\bm{w}_{ik} \in \mathbb{R}^p, \quad z^1_{ik},z^2_{ik} \in \mathbb{R}, \quad 1\leq k\leq K,1 \leq i \leq m,\\[0.5em]
& \bm{\Lambda} \in \mathbb{S}^p_+, \quad \bm{q} \in \mathbb{R}^p, \quad t\in\mathbb{R} ,\\[0.5em]
& \bm{x} \in X\subseteq\mathbb{R}^n.
\end{split}
\end{equation}
\end{small}
\end{theorem}
\begin{proof}
Recall the MGDRO model (\ref{eq:mgdro1}):
\begin{equation*}
\min \limits_{\bm{x} \in X} \max \limits_{P \in \mathcal{D}_M} E_{P}[h(\bm{x},\bm{\xi})- \min \limits_{1 \leq i \leq m} \theta_i \min \limits_{\bm{\xi}' \in Y_i} \phi_i(\bm{\xi},\bm{\xi}')],
\end{equation*}
and the ambiguity set (\ref{eq:ambiguity1}):
\begin{equation*}
   \mathcal{D}_M=
   \left\{
   P \in \mathcal{P}(\Xi)
   \middle\vert
   \begin{array}{lcl}
    & (E_{P}[\bm{\xi}]-\bm{\mu}_0)^{T}\bm{\Sigma}_0^{-1}(E_{P}[\bm{\xi}]-\bm{\mu}_0) \leq \gamma_1  \\
    & E_{P}[(\bm{\xi}-\bm{\mu}_0)(\bm{\xi}-\bm{\mu}_0)^{T}] \preceq \gamma_2\bm{\Sigma}_0\\
    & P(\bm{\xi} \in \Xi) = 1
   \end{array}
   \right\}.
\end{equation*}
With the same proof arguments as for Lemma \ref{lemma:dual}, the  MGDRO model (\ref{eq:mgdro1}) with ambiguity set (\ref{eq:ambiguity1}) can be equivalently reformulated as
\begin{equation}\label{eq:mgdro1:trac1}
\begin{split}
\min \limits_{\bm{x},\bm{\Lambda},\bm{q},t} \quad & t+\bm{\Lambda} \cdot (\gamma_2 \bm{\Sigma}_0 +\bm{\mu}_0 \bm{\mu}_0^{T})+ \sqrt{\gamma_1}||\bm{\Sigma}_0^{1/2} (\bm{q}+2\bm{\Lambda} \bm{\mu}_0)||_2 + \bm{q}^{T} \bm{\mu}_0 \\
s.t. \quad & h(\bm{x},\bm{\xi})-\bm{\xi}^{T}\bm{\Lambda} \bm{\xi} - \bm{q}^{T} \bm{\xi}-t \leq \min \limits_{1 \leq i \leq m} \theta_i \min \limits_{\bm{\xi}' \in Y_i} \phi_i(\bm{\xi},\bm{\xi}') \qquad \forall \bm{\xi} \in \Xi, \\
& \bm{\Lambda} \in \mathbb{S}^p_+, \quad \bm{q} \in \mathbb{R}^p, \quad t\in\mathbb{R} ,\\
& \bm{x} \in X\subseteq\mathbb{R}^n.
\end{split}
\end{equation}
The model (\ref{eq:mgdro1:trac1}) is a convex program with a semi-infinite constraint.
With the Assumption \ref{as:plinear}, the semi-infinite constraint is equivalent to a series of semi-infinite constraints:
\begin{equation}\label{eq:semiconsl}
\bm{a}_k(\bm{x})^{T} \bm{\xi}+b_k(\bm{x})-\bm{\xi}^{T}\bm{\Lambda} \bm{\xi} - \bm{q}^{T} \bm{\xi} - t \leq \theta_i \min \limits_{\bm{\xi}' \in Y_i} \phi_i(\bm{\xi},\bm{\xi}') \quad \forall \bm{\xi} \in \Xi \quad \forall k \quad \forall i.
\end{equation}
Given that 
$U_1 = Y_i = \{ \bm{\xi}_i + \bm{A}_i \bm{\zeta} | \bm{\zeta} \in Z_{1i} \}, U_2 = \Xi = \{ \bm{\xi}_0 + \bm{A} \bm{\zeta} | \bm{\zeta} \in Z_2 \}, \phi = \theta_i \phi_i$,
$f_k(\bm{\xi},(\bm{x},\bm{\Lambda},\bm{q},t))=\bm{a}_k(\bm{x})^{T} \bm{\xi} + b_k(\bm{x}) - \bm{\xi}^{T} \bm{\Lambda} \bm{\xi} - \bm{q}^{T} \bm{\xi}- t$
in Lemma \ref{lemma:gro}. It is obvious that  $f_k(\cdot,(\bm{x},\bm{\Lambda},\bm{q},t))$ is a closed concave function. According to Lemma \ref{lemma:gro}, a single semi-infinite constraint in (\ref{eq:semiconsl}) is equivalent to the existence of $\bm{v}_{ik}, \bm{w}_{ik} \in \mathbb{R}^p$ satisfying
\begin{equation}\label{eq:newcons}
\delta^*(\bm{v}_{ik}|Y_{i}) + \delta^*(\bm{w}_{ik}|\Xi) +(\theta_i \phi_i)^{**}(\bm{v}_{ik},-\bm{v}_{ik}) -f_{k*}(\bm{v}_{ik}+\bm{w}_{ik},(\bm{x},\bm{\Lambda},\bm{q},t))\leq 0, 
\end{equation}
where
\begin{equation*}
\begin{split}
\delta^*(\bm{v}_{ik}|Y_{i})& = \bm{\xi}_i^{T}\bm{v}_{ik} + \delta^*(\bm{A}_i^{T}\bm{v}_{ik}|Z_{1i}) \\
\delta^*(\bm{w}_{ik}|\Xi)& = \bm{\xi}_0^{T}\bm{w}_{ik}+ \delta^*(\bm{A}^{T}\bm{w}_{ik}|Z_2) \\
(\theta_i \phi_i)^{**}(\bm{v}_{ik},-\bm{v}_{ik})& = \theta_i \phi_i^{**}(\frac{\bm{v}_{ik}}{\theta_i},-\frac{\bm{v}_{ik}}{\theta_i}) \\
f_{k,*}(\bm{y},(\bm{x},\bm{\Lambda},\bm{q},t))
& = \inf \limits_{\bm{\xi} \in \mathbb{R}^p} \{\bm{y}^{T} \bm{\xi}-f_k(\bm{\xi},(\bm{x},\bm{\Lambda},\bm{q},t))\} \\
& = \inf \limits_{\bm{\xi} \in \mathbb{R}^p} \{\bm{\xi}^{T} \bm{\Lambda} \bm{\xi} + (\bm{y}+\bm{q}-\bm{a}_k(\bm{x}))^{T}\bm{\xi} + t -b_k(\bm{x}) \}.
\end{split}
\end{equation*} 
Hence, if we substitute these expressions in the inequality (\ref{eq:newcons}), it is equivalent to the following form:
\begin{equation*}
\begin{split}
& \bm{\xi}_i^{T}\bm{v}_{ik} + \bm{\xi}_0^{T}\bm{w}_{ik}+\delta^*(\bm{A}_i^{T}\bm{v}_{ik}|Z_{1i}) + \delta^*(\bm{A}^{T}\bm{w}_{ik}|Z_2)+ \theta_i \phi_i^{**}(\frac{\bm{v}_{ik}}{\theta_i},-\frac{\bm{v}_{ik}}{\theta_i}) \\
& \leq \bm{\xi}^{T} \bm{\Lambda} \bm{\xi} + (\bm{v}_{ik} + \bm{w}_{ik} +\bm{q}-\bm{a}_k(\bm{x}))^{T}\bm{\xi} + t -b_k(\bm{x}) \quad\forall \bm{\xi} \in \mathbb{R}^p,\\
\end{split}
\end{equation*}
Futhermore, it is equivalent to a linear matrix inequality:
\begin{equation*}
\begin{bmatrix}
      \bm{\Lambda} & \dfrac{1}{2} (\bm{v}_{ik} + \bm{w}_{ik} +\bm{q}-\bm{a}_k(\bm{x}))\\[1.5em]
      \dfrac{1}{2}(\bm{v}_{ik} + \bm{w}_{ik} +\bm{q}-\bm{a}_k(\bm{x}))^{T} & \begin{aligned}t-\bm{\xi}_i^{T}\bm{v}_{ik} - \bm{\xi}_0^{T}\bm{w}_{ik}- \theta_i \phi_i^{**}(\frac{\bm{v}_{ik}}{\theta_i},-\frac{\bm{v}_{ik}}{\theta_i}) \\ -b_k(\bm{x})-\delta^*(\bm{A}_i^{T}\bm{v}_{ik}|Z_{1i}) - \delta^*(\bm{A}^{T}\bm{w}_{ik}|Z_2) \end{aligned}\\
\end{bmatrix}  \succeq 0 
\end{equation*}
which can be rewritten as
\begin{equation*}
\begin{split}
& \left\{ \begin{aligned}
       &\begin{bmatrix}
      \bm{\Lambda} & \dfrac{1}{2} (\bm{v}_{ik} + \bm{w}_{ik} +\bm{q}-\bm{a}_k(\bm{x}))\\[1.5em]
      \dfrac{1}{2}(\bm{v}_{ik} + \bm{w}_{ik} +\bm{q}-\bm{a}_k(\bm{x}))^{T} & \begin{aligned}t-b_k(\bm{x})-\bm{\xi}_i^{T}\bm{v}_{ik} - \bm{\xi}_0^{T}\bm{w}_{ik} \\ - \theta_i \phi_i^{**}(\frac{\bm{v}_{ik}}{\theta_i},-\frac{\bm{v}_{ik}}{\theta_i})-z^1_{ik} -z^2_{ik} \end{aligned} \\
      \end{bmatrix}  \succeq 0 \\
      & \delta^*(\bm{A}_i^{T}\bm{v}_{ik}|Z_{1i}) \leq z^1_{ik} \\
      & \delta^*(\bm{A}^{T}\bm{w}_{ik}|Z_2) \leq z^2_{ik} \\
    \end{aligned} \right.
\end{split}
\end{equation*}
for simplicity. Then the result follows from (\ref{eq:mgdro1:trac1}).
\end{proof}


\subsubsection{MGDRO model with metric-based ambiguity set}
The DRO model with the ambiguity set $\mathcal{D}_{W}$ can be reformulated as a finite-dimensional convex program under some convexity assumptions, as shown in \cite{dro:wasserstein}. It is similarly operated for the MGDRO model. Combined with Lemma \ref{lemma:gro}, the reformulation of the MGDRO model with the ambiguity set $\mathcal{D}_W$ is derived in Theorem \ref{thm:mgdro_w}.
\begin{theorem}\label{thm:mgdro_w}
Suppose $r >0$ and Assumption \ref{as:plinear} are satisfied. $Y_i$ and $\Xi$ are defined as (\ref{eq:mcsandss}). Then, MGDRO model (\ref{eq:mgdro1}) with ambiguity set (\ref{eq:ambiguity_w}) can be reformulated as 
\begin{small}
\begin{equation}\label{eq:mgdrc_w}
\begin{aligned}
\min \limits_{\bm{x},\lambda,s_j,\bm{v}_{ijk},\bm{w}_{ijk}} \quad &\lambda r + \frac{1}{N} \sum_{j=1}^N s_j \\
s.t. \quad & \bm{\xi}_i^{T}\bm{v}_{ijk} + \delta^*(\bm{A}_i^{T}\bm{v}_{ijk}|Z_{1i}) + \bm{\xi}_0^{T}\bm{w}_{ijk}+ \delta^*(\bm{A}^{T}\bm{w}_{ijk}|Z_2)  \\
& \qquad - \hat{\bm{\xi}}_j^T(\bm{v}_{ijk}+\bm{w}_{ijk}- \bm{a}_k(\bm{x}))+ b_k(\bm{x}) - s_j + \theta_i \phi_i^{**}(\frac{\bm{v}_{ijk}}{\theta_i},-\frac{\bm{v}_{ijk}}{\theta_i})  \leq  0 \\
& ||\bm{v}_{ijk}+\bm{w}_{ijk}- \bm{a}_k(\bm{x})||_* \leq \lambda, \quad  1\leq j \leq N, 1 \leq i \leq m, 1 \leq k \leq K \\
& \lambda \geq 0, \bm{x} \in X
\end{aligned}
\end{equation}
\end{small}
\end{theorem}
\begin{proof}
Recall the MGDRO model (\ref{eq:mgdro1}):
\begin{equation*}
\min \limits_{\bm{x} \in X} \max \limits_{P \in \mathcal{D}_W} E_{P}[h(\bm{x},\bm{\xi})- \min \limits_{1 \leq i \leq m} \theta_i \min \limits_{\bm{\xi}' \in Y_i} \phi_i(\bm{\xi},\bm{\xi}')],
\end{equation*}
and the ambiguity set (\ref{eq:ambiguity_w}):
\begin{equation*}
\mathcal{D}_W = \{P \in \mathcal{M}(\Xi) \vert W(P,\hat{P}_N) \leq r \}.
\end{equation*}
The inner maximization problem can be rewiritten as
\begin{equation*}
\begin{aligned}
\max_{P,\Pi} \quad & E_P[h(\bm{x},\bm{\xi})- \min \limits_{1 \leq i \leq m} \theta_i \min \limits_{\bm{\xi}' \in Y_i} \phi_i(\bm{\xi},\bm{\xi}')] \\
s.t. \quad &\int_{\Xi^2} ||\bm{\xi} - \bm{\eta}|| \Pi(d\bm{\xi},d\bm{\eta})\leq r \\
& \left\{\begin{split} & \Pi \text{ is a joint distribution of }\bm{\xi} \text{ and } \bm{\eta} \\ & \text{with marginals } P \text{ and } \hat{P}_N, \text{ respectively.} \end{split} \right. \\
= \max_{P_j \in \mathcal{M}(\Xi)} \quad & \frac{1}{N} \sum_{j=1}^N \int_{\Xi} \left( h(\bm{x},\bm{\xi})- \min \limits_{1 \leq i \leq m} \theta_i \min \limits_{\bm{\xi}' \in Y_i}  \phi_i(\bm{\xi},\bm{\xi}') \right) P_j(d\bm{\xi}) \\
s.t. \quad & \frac{1}{N} \sum_{j=1}^N \int_{\Xi} ||\bm{\xi}-\hat{\bm{\xi}}_j|| P_j(d\bm{\xi}) \leq r \\
\end{aligned}
\end{equation*}
where $P_j$ is the conditional distributions of $\bm{\xi}$ given $\bm{\eta} = \hat{\bm{\xi}}_j, j =1,\ldots, N$. The equality follows from the formula of total probability. Then the Lagrangian dual of the inner maximization problem is easily obtained.
\begin{equation*}
\begin{split}
& \min \limits_{\lambda \geq 0} \max_{P_j \in \mathcal{M}(\Xi)} \left\{ \lambda r + \frac{1}{N} \sum_{j=1}^N \int_{\Xi} \left( h(\bm{x},\bm{\xi})- \min \limits_{1 \leq i \leq m} \theta_i \min \limits_{\bm{\xi}' \in Y_i}  \phi_i(\bm{\xi},\bm{\xi}')- \lambda ||\bm{\xi}-\hat{\bm{\xi}}_j|| \right)  P_j(d\bm{\xi}) \right\}\\
= & \min \limits_{\lambda \geq 0} \left\{\lambda r + \frac{1}{N} \sum_{j=1}^N \max_{P_j \in \mathcal{M}(\Xi)} \left\{\int_{\Xi} \left( h(\bm{x},\bm{\xi})- \min \limits_{1 \leq i \leq m} \theta_i \min \limits_{\bm{\xi}' \in Y_i} \phi_i(\bm{\xi},\bm{\xi}')- \lambda ||\bm{\xi}-\hat{\bm{\xi}}_j|| \right) P_j(d\bm{\xi}) \right\} \right\} \\
= & \min \limits_{\lambda \geq 0} \left\{\lambda r + \frac{1}{N} \sum_{j=1}^N \max_{\bm{\xi} \in \Xi} \left\{h(\bm{x},\bm{\xi})- \min \limits_{1 \leq i \leq m} \theta_i \min \limits_{\bm{\xi}' \in Y_i}  \phi_i(\bm{\xi},\bm{\xi}')- \lambda ||\bm{\xi}-\hat{\bm{\xi}}_j||) \right\} \right\} \\
= & \min \limits_{\lambda \geq 0} \quad \lambda r + \frac{1}{N} \sum_{j=1}^N s_j \\
& s.t. \quad h(\bm{x},\bm{\xi})- \min \limits_{1 \leq i \leq m} \theta_i \min \limits_{\bm{\xi}' \in Y_i} \phi_i(\bm{\xi},\bm{\xi}')- \lambda ||\bm{\xi}-\hat{\bm{\xi}}_j|| \leq s_j \quad \forall \bm{\xi} \in \Xi \quad \forall j \\
\end{split}
\end{equation*}
The strong duality is guaranteed by \cite[Proposition 3.4]{Shapiro}, the second equality follows from that all Dirac distributions $\delta_{\bm{\xi}},\bm{\xi} \in \Xi$ are contained in $\mathcal{M}(\Xi)$, and the last equality is a natural reformulation with introducing auxiliary variables $s_j$. Under the Assumption \ref{as:plinear}, the constraint
\begin{equation*}
h(\bm{x},\bm{\xi})- \min \limits_{1 \leq i \leq m} \theta_i \min \limits_{\bm{\xi}' \in Y_i} \phi_i(\bm{\xi},\bm{\xi}')- \lambda ||\bm{\xi}-\hat{\bm{\xi}}_j|| \leq s_j \quad \forall \bm{\xi} \in \Xi,
\end{equation*}
is rewritten as
\begin{equation}\label{eq:semicons2}
\bm{a}_k(\bm{x})^{T} \bm{\xi}+b_k(\bm{x})- \lambda ||\bm{\xi}-\hat{\bm{\xi}}_j|| - s_j  \leq  \theta_i \min \limits_{\bm{\xi}' \in Y_i} \phi_i(\bm{\xi},\bm{\xi} ') 
 \quad \forall \bm{\xi} \in \Xi \quad \forall i \quad \forall k .
\end{equation}
Given that 
$U_1 = Y_i = \{ \bm{\xi}_i + \bm{A}_i \bm{\zeta} | \bm{\zeta} \in Z_{1i} \}, U_2 = \Xi = \{ \bm{\xi}_0 + \bm{A} \bm{\zeta} | \bm{\zeta} \in Z_2 \}, \phi = \theta_i \phi_i$, 
$f_k(\bm{\xi},(\bm{x},s_j))=\bm{a}_k(\bm{x})^{T} \bm{\xi} + b_k(\bm{x}) - \lambda ||\bm{\xi}-\hat{\bm{\xi}}_j|| - s_j$
in Lemma \ref{lemma:gro}. It is obvious that  $f_k(\cdot,(\bm{x},s_j))$ is a closed concave function. According to Lemma \ref{lemma:gro}, a single semi-infinite constraint in (\ref{eq:semicons2}) is equivalent to the existence of $\bm{v}_{ijk}, \bm{w}_{ijk} \in \mathbb{R}^p$ satisfying
\begin{equation}\label{eq:newcons2}
\delta^*(\bm{v}_{ijk}|Y_{i}) + \delta^*(\bm{w}_{ijk}|\Xi) +(\theta_i \phi_i)^{**}(\bm{v}_{ijk},-\bm{v}_{ijk}) -f_{k*}(\bm{v}_{ijk}+\bm{w}_{ijk},(\bm{x},s_j))\leq 0, 
\end{equation}
where
\begin{equation*}
\begin{split}
\delta^*(\bm{v}_{ijk}|Y_{i})& = \bm{\xi}_i^{T}\bm{v}_{ijk} + \delta^*(\bm{A}_i^{T}\bm{v}_{ijk}|Z_{1i}) \\
\delta^*(\bm{w}_{ijk}|\Xi)& = \bm{\xi}_0^{T}\bm{w}_{ijk}+ \delta^*(\bm{A}^{T}\bm{w}_{ijk}|Z_2) \\
(\theta_i \phi_i)^{**}(\bm{v}_{ijk},-\bm{v}_{ijk})& = \theta_i \phi_i^{**}(\frac{\bm{v}_{ijk}}{\theta_i},-\frac{\bm{v}_{ijk}}{\theta_i}) \\
f_{ijk*}(\bm{y},\bm{x},s_j) & = \inf \limits_{\bm{\xi}\in \mathbb{R}^p} \left\{\bm{y}^T\bm{\xi} - \bm{a}_k(\bm{x})^{T} \bm{\xi}- b_k(\bm{x})+ \lambda ||\bm{\xi}-\hat{\bm{\xi}}_j|| + s_j \right\} \\
& = \inf \limits_{\bm{\xi}\in \mathbb{R}^p} \left\{\lambda ||\bm{\xi}|| + (\bm{y}- \bm{a}_k(\bm{x}))^T (\bm{\xi}+ \hat{\bm{\xi}}_j) - b_k(\bm{x}) + s_j \right\} \\
& = \inf \limits_{\bm{\xi}\in \mathbb{R}^p} \left\{\lambda ||\bm{\xi}|| + (\bm{y}- \bm{a}_k(\bm{x}))^T\bm{\xi}\right\} + (\bm{y}- \bm{a}_k(\bm{x}))^T \hat{\bm{\xi}}_j - b_k(\bm{x}) + s_j  \\
& = \left\{ \begin{split} &(\bm{y}- \bm{a}_k(\bm{x}))^T \hat{\bm{\xi}}_j - b_k(\bm{x}) + s_j  \quad  \text{if} \quad ||\bm{y}- \bm{a}_k(\bm{x})||_* \leq \lambda \\  &-\infty \quad \text{otherwise} \end{split} \right.
\end{split}
\end{equation*}
Hence, if we substitute these expressions in the inequality (\ref{eq:newcons2}), it is equivalent to the following form:
\begin{equation*}
\left\{
\begin{split}
& \bm{\xi}_i^{T}\bm{v}_{ijk} + \delta^*(\bm{A}_i^{T}\bm{v}_{ijk}|Z_{1i}) + \bm{\xi}_0^{T}\bm{w}_{ijk}+ \delta^*(\bm{A}^{T}\bm{w}_{ijk}|Z_2) \\
& \qquad  - \hat{\bm{\xi}}_j^T(\bm{v}_{ijk}+\bm{w}_{ijk}- \bm{a}_k(\bm{x})) + b_k(\bm{x}) - s_j  + \theta_i \phi_i^{**}(\frac{\bm{v}_{ijk}}{\theta_i},-\frac{\bm{v}_{ijk}}{\theta_i}) \leq  0 \\
& ||\bm{v}_{ijk}+\bm{w}_{ijk}- \bm{a}_k(\bm{x})||_* \leq \lambda
\end{split}
\right.
\end{equation*}
Then the inner maximization problem is equivalently reformulated as 
\begin{small}
\begin{equation}\label{eq:inner_re}
\begin{aligned}
\min \limits_{\lambda,s_j,\bm{v}_{ijk},\bm{w}_{ijk}} \quad &\lambda r + \frac{1}{N} \sum_{j=1}^N s_j \\
s.t. \quad & \bm{\xi}_i^{T}\bm{v}_{ijk} + \delta^*(\bm{A}_i^{T}\bm{v}_{ijk}|Z_{1i}) + \bm{\xi}_0^{T}\bm{w}_{ijk}+ \delta^*(\bm{A}^{T}\bm{w}_{ijk}|Z_2)  \\
& \qquad - \hat{\bm{\xi}}_j^T(\bm{v}_{ijk}+\bm{w}_{ijk}- \bm{a}_k(\bm{x}))+ b_k(\bm{x}) - s_j + \theta_i \phi_i^{**}(\frac{\bm{v}_{ijk}}{\theta_i},-\frac{\bm{v}_{ijk}}{\theta_i})  \leq  0 \\
& ||\bm{v}_{ijk}+\bm{w}_{ijk}- \bm{a}_k(\bm{x})||_* \leq \lambda, \quad  1\leq j \leq N, 1 \leq i \leq m, 1 \leq k \leq K \\
& \lambda \geq 0.
\end{aligned}
\end{equation}
\end{small}
The result follows from combining (\ref{eq:inner_re}) with the outer minimization.
\end{proof}

Notice that the computations involving $Z_{1i}$, $Z_2$ and $\phi_i$ are all separated in the two reformulations (\ref{eq:mgdrc1}) and (\ref{eq:mgdrc_w}). Therefore, the MGDRO models are SDP representable with the choices of $Z_{1i}$, $Z_2$ and $\phi_i$ introduced in \cite{gdro}, as shown in Table \ref{table:setanddf}, when $q_1 = 1,2,\infty$. In addition, the norm used in Wasserstein distance should be taken as $1-$, $2-$ or $\infty$-norm.

\begin{table}[thb]
\renewcommand\arraystretch{1.5}
\centering
\caption{Choices of sets and distance functions}\label{table:setanddf}
\begin{threeparttable}
\begin{tabular}{ll}
\toprule[1.5pt]
Set & Relative conjugate function \\
\noalign{\smallskip}\hline\noalign{\smallskip}
$Z_i = \{\bm{\zeta} | \bm{C}_i \bm{\zeta} \leq \bm{d}_i\}$ 
& $\delta^*(\bm{u}|Z_i) = \min \limits_{\bm{y} \geq 0} \{\bm{d}_i^T \bm{y} | \bm{C}_i^T \bm{y} = \bm{u}\} $ \\
$Z_i = \{\bm{\zeta} | ||\bm{\zeta}||_{q_1} \leq \sqrt{\bar{\gamma}_i}\}$ 
& $\delta^*(\bm{u}|Z_i) = \sqrt{\bar{\gamma}_i} ||\bm{u}||_{q_2}$ \\
$Z_2 = \mathbb{R}^L$ 
& $\delta^*(\bm{u}|Z_2) = \left\{ \begin{aligned} & 0 \quad \text{if } \bm{u}=0, \\ & \infty \quad \text{otherwise}. \end{aligned}\right.$ \\
\midrule[1pt]
Distance function & Relative conjugate function \\
\midrule[1pt]
$\phi(\bm{\xi},\bm{\xi}') = ||\bm{\xi}-\bm{\xi}'||_{q_1}$ 
& $\theta \phi^{**}(\frac{\bm{v}}{\theta},-\frac{\bm{v}}{\theta}) = \left\{
      \begin{aligned}
      & 0 \quad \text{if} \quad ||\bm{v}||_{q_2} \leq \theta , \\
      & \infty \quad \text{otherwise},
      \end{aligned}
      \right.$  \\
$\phi(\bm{\xi},\bm{\xi}') = ||\bm{\xi}-\bm{\xi}'||_{q_1} ^2$
& $\theta \phi^{**}(\frac{\bm{v}}{\theta},-\frac{\bm{v}}{\theta}) = \frac{1}{4\theta} ||\bm{v}||_{q_2}^2$ \\
\bottomrule[1.5pt]
\end{tabular}
\begin{tablenotes}
\item[1] \footnotesize $1/q_1 + 1/q_2 = 1$.
\item[2] When $q_2 = 1,2,\infty$, the relative conjugate function is SDP representable.
\end{tablenotes}
\end{threeparttable}
\end{table}

\section{Application in multi-product newsvendor problem}\label{sec:app}
\subsection{Multi-product newsvendor problem}
In this section, the MGDRO models are applied to a multi-product newsvendor problem with multimodal demands. The multi-probuct newsvendor problem is an important basic model in the inventory management. The seller needs to determine the order quantities with fixed prices and uncertain demands. Due to the uncertainty of demands, robust optimization and distributionally robust optimization are powerful tools for multi-product newsvendor problem, see for example \cite{mpnp:ro,mpnp:dro,np:dro:wass}. Especially, Hanasusanto et al. \cite{dro:multimodal} focused on multi-product newsvendor problem with multimodal demand distributions. Hence, we adopted the same problem studied by Hanasusanto et al. \cite{dro:multimodal} and some of our experimental settings also followed those of them.

Let $\bm{x} \in \mathbb{R}^n$ be the vector of order quantities, and $\bm{\xi} \in \mathbb{R}^n$ be the random vector of demands. Accordingly, $\min(\bm{x},\bm{\xi})$ is the vector of sales quantities, where ``$\min$" denotes component-wise minimization. The seller orders the products at wholesale prices and sells them at retail prices. When $x_i<\xi_i$, unsatisfied demand incurs a stock-out cost. On the other hand, unsold stock $x_i-\xi_i$ is cleared at the salvage price when $x_i>\xi_i$. Let $\bm{c}$, $\bm{v}$, $\bm{b}$ and $\bm{g}$ denote the vectors of the wholesale prices, retail prices, stock-out costs and salvage prices, respectively. Spontaneously, assume that $\bm{c}<\bm{v}$ and $\bm{g} < \bm{v}$. Hence, the loss function can be presented as
\begin{equation}\label{eq:loss}
\begin{aligned}
L(\bm{x},\bm{\xi}) & = \bm{c}^T\bm{x} - \bm{v}^T\min(\bm{x},\bm{\xi}) - \bm{g}^T(\bm{x}-\min(\bm{x},\bm{\xi})) + \bm{b}^T(\bm{\xi}-\min(\bm{x},\bm{\xi})) \\
& = \bm{d}^T\bm{x} + \bm{b}^T\bm{\xi} + \bm{h}^T\max(\bm{x}-\bm{\xi},0),
\end{aligned}
\end{equation}
where $\bm{d}=\bm{c}-\bm{v}-\bm{b}$ and $\bm{h}=\bm{v}+\bm{b}-\bm{g}$.

We assume that the decision maker is risk-averse, and use Conditional Value-at-Risk (CVaR) to measure the risk of the loss.
\begin{equation}\label{eq:cvar}
\text{CVaR}_{\varepsilon}(L(\bm{x},\bm{\xi})) = \min \limits_{\beta \in \mathbb{R}} \left( \beta + \frac{1}{\varepsilon} E_P[( L(\bm{x},\bm{\xi})-\beta)^+]\right).
\end{equation}
Hence, the MGDRO model is rewritten as 
\begin{equation}\label{eq:mnmgdro}
\min \limits_{\bm{x} \in \mathbb{R}^n_+,\beta \in \mathbb{R}} 
\sup \limits_{P \in \mathcal{D}} \left( \beta + \frac{1}{\varepsilon}  E_P[(L(\bm{x},\bm{\xi}) - \beta)^+ -\min \limits_{1 \leq i \leq m} \theta_i \min \limits_{\bm{\xi}' \in Y_i} \phi_i(\bm{\xi},\bm{\xi}')]\right),
\end{equation}
where the ambiguity set $\mathcal{D}$ can be taken as (\ref{eq:ambiguity1}) or (\ref{eq:ambiguity_w}), and the loss function $L$ is taken as (\ref{eq:loss}).

Notice that $\bm{h}>0$ in (\ref{eq:loss}). Then, 
\begin{equation}\label{eq:exobj}
\bm{h}^T \max (\bm{x}-\bm{\xi},0) = \sum \limits_{i=1}^n h_i \max(x_i-\xi_i,0) = \max \limits_{1\leq k \leq 2^n} \bm{h}_k^T (\bm{x}-\bm{\xi}),
\end{equation}
where $\bm{h}_k = \bm{I}_k \bm{h}$, $\bm{I}_k \in \mathbb{R}^{n \times n}, \bm{I}_k^{(ii)}=1$ if and only if $i \in \mathcal{I}_k$, $\bm{I}_k^{(ij)}=0$ for any other $i,j$, and $\{\mathcal{I}_k,1\leq k \leq 2^n\}$ are all subsets of $\{1,2,\dots,n\}$. Hence, the objective function in (\ref{eq:mnmgdro}) satisfies Assumption \ref{as:plinear}.

\subsection{Numerical experiments}
In this subsection, some experiment settings follow Hanasusanto et al. \cite{dro:multimodal}. Let $n=3$, the wholesale price, retail price, salvage price and stockout cost of each product be 5, 10, 1 and 2.5, respectively. The risk preference of the decision maker is presented as the parameter in the CVaR measure (\ref{eq:cvar}) that $\varepsilon=0.05$.
The true probability distribution is an equiprobable mixture of $m$ truncated normal distributions $N(\bm{\mu}_i,\bm{\Sigma}_i,\rho_i)$, where $N(\bm{\mu}_i,\bm{\Sigma}_i,\rho_i)$ is obtained by truncating the normal distribution $N(\bm{\mu}_i,F_n(\rho_i^2)/F_{n+2}(\rho_i^2) \bm{\Sigma}_i)$ outside of the ellipsoid $\{\bm{\xi}| (\bm{\xi}-\bm{\mu}_i)^T \bm{\Sigma}_i^{-1} (\bm{\xi}-\bm{\mu}_i) \leq \rho_i^2 F_n(\rho_i^2)/F_{n+2}(\rho_i^2) \}$, and $F_k$ is the cumulative distribution function of $\chi^2$ distribution with the degree of freedom $k$. 

We consider bimodal and trimodal cases, i.e. $m=2,3$. The parameters of the normal distributions are shown in Table \ref{table:bim} and \ref{table:trim} respectively. Notice that the means in trimodal case are twice as much as those in bimodal case. It aims to exhibit more apparent multimodal characteristics of the data by making the 3$\sigma$ intervals of probability distribution for fixed product in each state separate from each other. In addtion, the demands of any two products are correlated with correlation coefficients of $50\%$ in each state, and $\rho_i = (F_n^{-1}(99\%))^{1/2},i=1,\dots,m$. The scatter plots of sample points with such bimodal and trimodal distributions are shown in Figure \ref{fig:multimodal}, where the multimodality of data is visual .

\begin{table}[th]
\renewcommand\arraystretch{1.5}
\caption{Marginal moments of multimodal distribution}\label{table:mmoments}
\centering
\begin{subtable}{.45\linewidth}
\centering
\caption{Bimodal case}\label{table:bim}
\resizebox{!}{1.75cm}{
\begin{tabular}{ccccc}
\toprule[1.5pt]\noalign{\smallskip}
State & Parameters & Product1 & Product2 & Product3 \\
\noalign{\smallskip}\midrule[1pt]\noalign{\smallskip}
\multirow{2}*{1} & Mean & 15 & 30 & 45  \\
~ & Variance & 25 & 25 & 25 \\
\noalign{\smallskip}\midrule[1pt]\noalign{\smallskip}
\multirow{2}*{2} & Mean & 45 & 30 & 15  \\
~ & Variance & 25 & 25 & 25 \\
\noalign{\smallskip}\bottomrule[1.5pt]
\end{tabular}
}
\end{subtable}
\begin{subtable}{.45\linewidth}
\centering
\caption{Trimodal case}\label{table:trim}
\resizebox{!}{2.5cm}{
\begin{tabular}{ccccc}
\toprule[1.5pt]\noalign{\smallskip}
State & Parameters & Product1 & Product2 & Product3 \\
\noalign{\smallskip}\midrule[1pt]\noalign{\smallskip}
\multirow{2}*{1} & Mean & 30 & 60 & 90  \\
~ & Variance & 25 & 25 & 25 \\
\noalign{\smallskip}\midrule[1pt]\noalign{\smallskip}
\multirow{2}*{2} & Mean & 60 & 90 & 30  \\
~ & Variance & 25 & 25 & 25 \\
\noalign{\smallskip}\midrule[1pt]\noalign{\smallskip}
\multirow{2}*{3} & Mean & 90 & 30 & 60  \\
~ & Variance & 25 & 25 & 25 \\
\noalign{\smallskip}\bottomrule[1.5pt]
\end{tabular}
}
\end{subtable} 
\end{table}  

\begin{figure}[tbhp]
\centering
\begin{subfigure}{.45\textwidth}
\centering  
\includegraphics[scale=0.45]{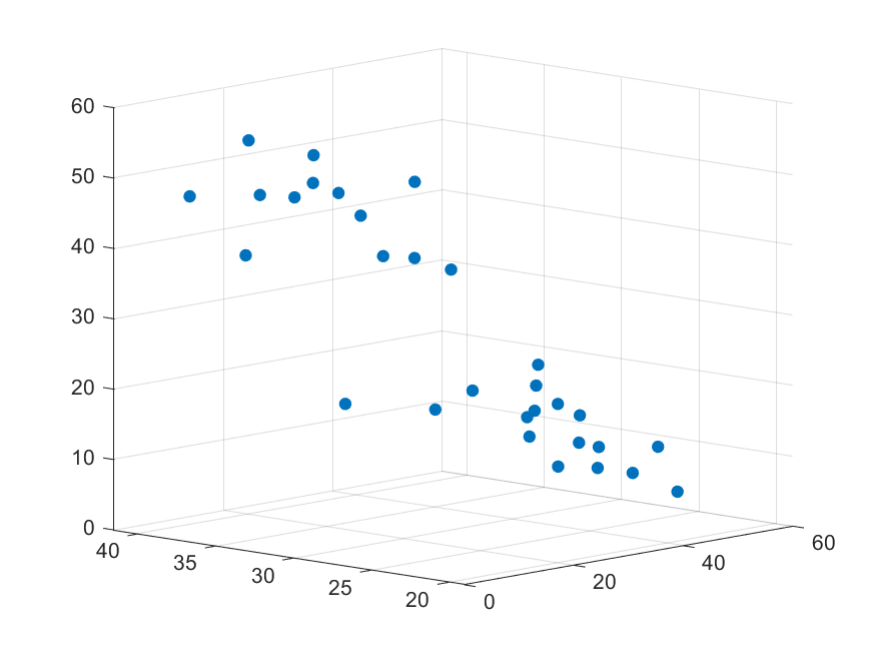}
\caption{Bimodal}\label{fig:bimodal}
\end{subfigure}
\begin{subfigure}{.45\textwidth}
\centering  
\includegraphics[scale=0.45]{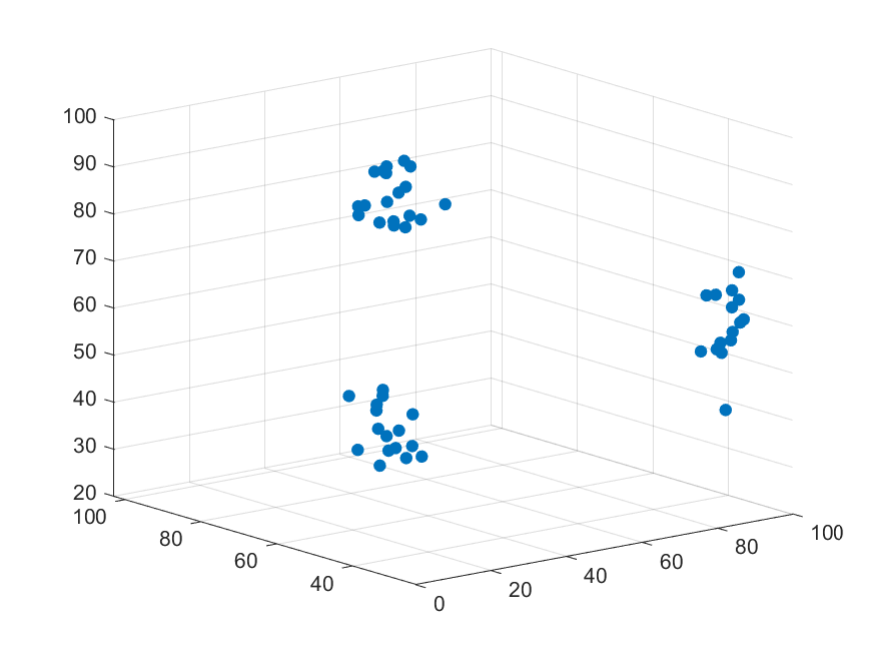}
\caption{Trimodal}\label{fig:trimodal}
\end{subfigure}
\caption{Scatter plot of sample points with multimodal distribution}\label{fig:multimodal}
\end{figure}

\begin{table}[tbhp]
\renewcommand\arraystretch{1.5}
\centering
\caption{Models in numerical experiments}\label{table:models}
\resizebox{\textwidth}{40mm}{
\begin{threeparttable}
\begin{tabular}{cccc}
\hline\noalign{\smallskip}
notation & model & sample space $\Xi$ & core set(s) $Y$($Y_i$)  \\
\noalign{\smallskip}\hline\noalign{\smallskip}
SP & $E_{\hat{P}_N}[h(\bm{x},\bm{\xi})]$ & $\{\hat{\bm{\xi}}_i,i=1,\ldots,N\}$  & - \\ 
DRO-M1 & $\min \limits_{\bm{x} \in X} \max \limits_{P \in \mathcal{D}_M} E_{P}[h(\bm{x},\bm{\xi})]$ & $\mathbb{R}^p$ & -  \\
DRO-M2($\gamma$) & $\min \limits_{\bm{x} \in X} \max \limits_{P \in \mathcal{D}_M} E_{P}[h(\bm{x},\bm{\xi})]$ & $\mathcal{E}(\bm{\mu}_0,\bm{\Sigma}_0,\gamma)$ & -  \\
DRO-W1 & $\min \limits_{\bm{x} \in X} \max \limits_{P \in \mathcal{D}_W} E_{P}[h(\bm{x},\bm{\xi})]$ & $\mathbb{R}^p$ & -  \\
DRO-W2($\gamma$) & $\min \limits_{\bm{x} \in X} \max \limits_{P \in \mathcal{D}_W} E_{P}[h(\bm{x},\bm{\xi})]$ & $\mathcal{E}(\bm{\mu}_0,\bm{\Sigma}_0,\gamma)$ & -  \\
MGDRO-M1 & $\min \limits_{\bm{x} \in X} \max \limits_{P \in \mathcal{D}_M} E_{P}[h(\bm{x},\bm{\xi})- \min \limits_{1 \leq i \leq m} \theta_i \min \limits_{\bm{\xi}' \in Y_i} \phi_i(\bm{\xi},\bm{\xi}')]$ & $\mathbb{R}^p$ & $\mathcal{E}(\bm{\mu}_0,\bm{\Sigma}_0,\bar{\gamma}_{1i})$  \\
MGDRO-M2 & $\min \limits_{\bm{x} \in X} \max \limits_{P \in \mathcal{D}_M} E_{P}[h(\bm{x},\bm{\xi})- \min \limits_{1 \leq i \leq m} \theta_i \min \limits_{\bm{\xi}' \in Y_i} \phi_i(\bm{\xi},\bm{\xi}')]$ & $\mathcal{E}(\bm{\mu}_0,\bm{\Sigma}_0,\bar{\gamma}_2)$ & $\mathcal{E}(\bm{\mu}_0,\bm{\Sigma}_0,\bar{\gamma}_{1i})$  \\
MGDRO-W1 & $\min \limits_{\bm{x} \in X} \max \limits_{P \in \mathcal{D}_W} E_{P}[h(\bm{x},\bm{\xi})- \min \limits_{1 \leq i \leq m} \theta_i \min \limits_{\bm{\xi}' \in Y_i} \phi_i(\bm{\xi},\bm{\xi}')]$ & $\mathbb{R}^p$ & $\mathcal{E}(\hat{\bm{\mu}}_i,\hat{\bm{\Sigma}}_i,\bar{\gamma}_{1i})$ \\
MGDRO-W2 & $\min \limits_{\bm{x} \in X} \max \limits_{P \in \mathcal{D}_W} E_{P}[h(\bm{x},\bm{\xi})- \min \limits_{1 \leq i \leq m} \theta_i \min \limits_{\bm{\xi}' \in Y_i} \phi_i(\bm{\xi},\bm{\xi}')]$ & $\mathcal{E}(\bm{\mu}_0,\bm{\Sigma}_0,\bar{\gamma}_2)$ & $\mathcal{E}(\hat{\bm{\mu}}_i,\hat{\bm{\Sigma}}_i,\bar{\gamma}_{1i})$ \\
\noalign{\smallskip}\hline
\end{tabular}
\begin{tablenotes}
\item [1] $\mathcal{E}(\bm{\mu},\bm{\Sigma},\gamma) = \{\bm{\xi} \in \mathbb{R}^p |(\bm{\xi}-\bm{\mu})^{T}\bm{\Sigma}^{-1}(\bm{\xi}-\bm{\mu}) \leq \gamma \}$.
\item [2] Especially, GDRO-M1 and GDRO-M2 represent the corresponding GDRO models which only use one core set.
\end{tablenotes}
\end{threeparttable}}
\end{table}

All models involving the experiments are listed in Table \ref{table:models}, where the SP model is the stochastic programming model\cite{sp} using the emprical distribution approximate the true probability distribution. As shown in Table \ref{table:models}, we consider the bounded and unbounded sample spaces of DRO and MGDRO models for comparison, respectively. For convenience, ellipsoidal core sets and sample spaces are adopted, where $\mathcal{E}(\bm{\mu},\bm{\Sigma},\gamma) = \{\bm{\xi} \in \mathbb{R}^p |(\bm{\xi}-\bm{\mu})^{T}\bm{\Sigma}^{-1}(\bm{\xi}-\bm{\mu}) \leq \gamma \}$ denotes the ellipsoid with parameters $\bm{\mu},\bm{\Sigma},\gamma$. Besides, $\bm{\mu}_0$ and $\bm{\Sigma}_0$ are sample mean and sample covariance matrix, $\hat{\bm{\mu}}_i$ is the center found by the clustering algorithm, and $\hat{\bm{\Sigma}}_i$ is the covariance matrix of all samples belonging to class $i$. Accordingly, $\bar{\gamma}_1$ ($\bar{\gamma}_2$) is the smallest $\gamma$ such that $50\%$ ($100\%$) samples are included in $\mathcal{E}(\bm{\mu}_0,\bm{\Sigma}_0,\gamma)$, while $\bar{\gamma}_{1i}$ is the smallest $\gamma$ such that $50\%$ samples of class $i$ are included in $\mathcal{E}(\hat{\bm{\mu}}_i,\hat{\bm{\Sigma}}_i,\gamma)$. Besides, the distance functions $\phi_i$ in MGDRO models are all taken as $||\bm{\xi}-\bm{\xi}'||_2$. 

The $\theta$s in MGDRO models and the radii of the Wasserstein balls are determined via a 5-fold cross validation. Partition $\{\hat{\bm{\xi}}_1,\ldots,\hat{\bm{\xi}}_N\}$ into five subsets $S_i,i=1,\ldots,5$ with approximately equal cardinality $N/5$. Each time, one subset $S_i$ is chosen as the validation dataset and the remaining four subsets are merged as a training dataset. Using only samples in the training dataset, the MGDRO model is solved with an alternative parameter value to obtain $x^{*(-i)}$. The validation dataset is used to estimate the out-of-sample performance of $x^{*(-i)}$ via sample average approximation. After all subsets successively play the role of the validation dataset, the average of five out-of-sample performances is used to measure the performance of the certain parameter value. The best parameter value among all alternative values is chosen in the numerical experiments.

The out-of-sample CVaR value is adopted to measure the performance of all models. We repeated each group of experiments 100 times and computed the average and variance of the out-of-sample CVaR values for each configuration.

\subsubsection{Comparison of different models with moment-based ambiguity set}
In this subsection, we focus on the moment-based ambiguity set. The MGDRO models are compared with SP, DRO and GDRO models. 
The numerical results are shown in Table \ref{table:multimodal}.

\begin{table}[th]
\renewcommand\arraystretch{1.5}
\caption{Out-of-sample CVaR values of moment-based framework}\label{table:multimodal}
\centering
\begin{subtable}{.45\linewidth}
\centering
\caption{Bimodal case}\label{table:bimodal}
\resizebox{!}{3.5cm}{
\begin{tabular}{ccc}
\toprule[1.5pt]
Model & Mean & Variance\\
\midrule[1pt]
SP & -211.7288 &  97.6492  \\
DRO-M1 & 15.6511 & 869.8281 \\
DRO-M2($\bar{\gamma}_2$) & -180.1215 & 290.6783   \\
DRO-M2($\bar{\gamma}_1$) & -208.6063  & 73.0621  \\
\midrule[1pt]	
GDRO-M1 & -208.7806 & 72.1910   \\
GDRO-M2 & -208.2579 & 83.8964   \\
MGDRO-M1 & -215.7776 & 59.2022   \\
MGDRO-M2 & -213.3395 & 61.7546 \\
\bottomrule[1.5pt]
\end{tabular}
}
\end{subtable}
\begin{subtable}{.45\linewidth}
\centering
\caption{Trimodal case}\label{table:trimodal}
\resizebox{!}{3.5cm}{
\begin{tabular}{ccc}
\toprule[1.5pt]
Model & Mean & Variance\\
\midrule[1pt]
SP & -521.8297 & 153.8776   \\
DRO-M1 & 90.7669 & 2041.1434 \\
DRO-M2($\bar{\gamma}_2$) & -419.2496 & 1547.9562   \\
DRO-M2($\bar{\gamma}_1$) & -509.5733  & 329.6365  \\
\midrule[1pt]	
GDRO-M1 & -509.7006 & 460.2741   \\
GDRO-M2 & -511.2113 & 383.4096   \\
MGDRO-M1 & -526.1398 & 79.9519    \\
MGDRO-M2 & -526.9204 & 87.1452  \\
\bottomrule[1.5pt]
\end{tabular}
}
\end{subtable} 
\end{table}  

MGDRO models outperform GDRO and DRO models significantly, whether from the perspective of the mean or variance. It demonstrates that the MGDRO models are able to make full use of the information of the multimodal structure and capture the structural characteristics accurately. Therefore, the MGDRO models reduce the degree of conservatism greatly and maintain high robustness regarding the set of samples simultaneously. Compared with the DRO models with the same sample space, the GDRO models with only one core set still shows improvement, especially the improvement regarding conservatism, implicated by smaller mean CVaR values. Therefore, the construction of the core sets, even only one core set, makes great sense of reducing the degree of conservatism. 

The SP model shows better performance than DRO and GDRO models, due to utilizing the sample data directly and retaining the structural information. It indicates that the structural information plays a decisive role in modelling with data from multimodal distribution, which is much more important than the moment information. The MGDRO models provide smaller means and variances of the out-of-sample CVaR values than those of SP model, indicating that the ambiguity set with moment information plays an important role as well. Hence, the MGDRO models combining the structural and moment information performs best in the sense of out-of-sample CVaR value.

\subsubsection{Comparison between moment-based and metric-based ambiguity set}
In this subsection, the contrast of the moment-based and metric-based ambiguity sets are investigated and shown.
The numerical results are shown in Table \ref{table:two_ambiguity_set}.

\begin{table}[th]
\renewcommand\arraystretch{1.5}
\caption{Out-of-sample CVaR values}\label{table:two_ambiguity_set}
\centering
\begin{subtable}{.45\linewidth}
\centering
\caption{Bimodal case}
\resizebox{!}{3.5cm}{
\begin{tabular}{ccc}
\toprule[1.5pt]
Model & Mean & Variance\\
\midrule[1pt]
SP & -208.5797 & 76.8787   \\
DRO-M1 & 15.4570 & 607.7011 \\
DRO-M2($\bar{\gamma}_2$) & -175.3580 & 325.4744   \\
DRO-W1 & -208.5659 & 70.9747 \\
DRO-W2($\bar{\gamma}_2$) & -207.9194 & 77.3860   \\
\midrule[1pt]
MGDRO-M1 & -210.3320 &  61.4321  \\
MGDRO-M2 & -210.4315 & 60.3978 \\
MGDRO-W1 & -208.5117 & 70.9678  \\
MGDRO-W2 & -208.4946 & 69.7494  \\
\bottomrule[1.5pt]
\end{tabular}
}
\end{subtable}
\begin{subtable}{.45\linewidth}
\centering
\caption{Trimodal case}
\resizebox{!}{3.5cm}{
\begin{tabular}{ccc}
\toprule[1.5pt]
Model & Mean & Variance\\
\midrule[1pt]
SP & -523.0786 &  142.5879  \\
DRO-M1 & 90.4633 & 2032.8589 \\
DRO-M2($\bar{\gamma}_2$) & -421.5343 &  1643.3484  \\
DRO-W1 & -523.0786 & 142.5879 \\
DRO-W2($\bar{\gamma}_2$) & -524.7676 & 125.8719  \\
\midrule[1pt]
MGDRO-M1 & -526.9301 & 100.2487   \\
MGDRO-M2 & -527.6518 & 107.5988 \\
MGDRO-W1 & -523.7801 & 145.7764  \\
MGDRO-W2 & -526.6669 & 109.5437  \\
\bottomrule[1.5pt]
\end{tabular}
}
\end{subtable} 
\end{table}  

The model DRO-W outperforms DRO-M, while the model MGDRO-M outperforms MGDRO-W. It is intuitive because the Wasserstein-distance-based ambiguity set leverages all original sample data, while the moment-based ambiguity set only includes the mean and convariance matrix information. Hence, DRO-W is able to capture the multimodal characteristics, thereby performs much better than DRO-M. 
It is also for this reason that the MGDRO framework shows much more improvement for the DRO model with moment-based ambiguity set than that with Wasserstein-distance-based ambiguity set. In addition, the fact that MGDRO-M outperforms MGDRO-W reinforces the effect of the moment information when making full use of the structual properties via core sets.

Note that whether the sample space is bounded or unbounded has little impact on the performance of the DRO and MGDRO models with metric-based ambiguity set. In order to attain better robustness, the whole space is more appropriate as sample space in DRO and MGDRO models with metric-based ambiguity set.

\subsubsection{Comparison with other multimodal ambiguity set}
Hanasusanto et al. \cite{dro:multimodal} formulated a distributionally robust multi-product newsvendor model with a multimodal ambiguity set, denoted as MDRO in the following.
\begin{equation*}
\min \limits_{\bm{x} \in \mathbb{R}^n_+,\beta \in \mathbb{R}} 
\sup \limits_{P \in \mathcal{D}_H} \left( \beta + \frac{1}{\varepsilon}  E_P[(L(\bm{x},\bm{\xi}) - \beta)^+]\right),
\end{equation*}
where
\begin{equation*}
   \mathcal{D}_H=
   \left\{
     \sum \limits_{i=1}^m p_i P_i
   \middle\arrowvert
   \begin{array}{lcl}
    & \sum \limits_{i=1}^m p_i = 1, \sum \limits_{i=1}^m (p_i-\hat{p}_i)/p_i \leq \delta_p, p_i \geq 0, P_i \in \mathcal{P}(\Xi_i) \\
    & \wideubar{\Omega}_i \leq \int_{\Xi_i} [\bm{\xi}^T 1]^T [\bm{\xi}^T 1] P_i(d\bm{\xi}) \leq \bar{\Omega}_i \quad \forall i = 1,\ldots,m
   \end{array}
   \right\},
\end{equation*}
\begin{equation*}
\begin{aligned}
& \wideubar{\Omega}_i = \begin{bmatrix} (1-\tau)^2(\hat{\bm{\Sigma}}_i+\hat{\bm{\mu}}_i \hat{\bm{\mu}}_i^T) & (1-\tau)\hat{\bm{\mu}}_i \\ (1-\tau)\hat{\bm{\mu}}_i^T & 1 \end{bmatrix},\\
& \bar{\Omega}_i = \begin{bmatrix} (1+\tau)^2(\hat{\bm{\Sigma}}_i+\hat{\bm{\mu}}_i \hat{\bm{\mu}}_i^T) & (1+\tau)\hat{\bm{\mu}}_i \\ (1+\tau)\hat{\bm{\mu}}_i^T & 1 \end{bmatrix},\\
\end{aligned}
\end{equation*}
and $\Xi_i$ is the support set under the state $i$.

Hence, we also compared MGDRO framework with the MDRO model. In the numerical experiments, $\hat{p}=\bm{e}_m/m $, $\delta_p = 0$, $\tau = 0.1$, and $\Xi_i$ is chosen as the ellipsoid $\mathcal{E}(\hat{\bm{\mu}}_i,\hat{\bm{\Sigma}}_i,\bar{\gamma}_{2i})$, where $\bar{\gamma}_{2i}$ is the smallest $\gamma$ such that all samples of class $i$ obtained by the clustering algorithm are included in $\mathcal{E}(\hat{\bm{\mu}}_i,\hat{\bm{\Sigma}}_i,\gamma)$. Besides, we also consider the special case when the true support sets and moments are known, i.e. $\hat{\bm{\mu}}_i = \bm{\mu}_i, \hat{\bm{\Sigma}}_i = \bm{\Sigma}_i, \bar{\gamma}_{2i} = \rho_i^2 F_n(\rho_i^2)/F_{n+2}(\rho_i^2)$, therefore it is regardless of the set of samples. The numerical results are shown in Table \ref{table:comp_Hana}, where MDRO-D denotes the data-driven MDRO model, and MDRO-T denotes the MDRO model with true support sets and moments.

\begin{table}[th]
\renewcommand\arraystretch{1.5}
\caption{Out-of-sample CVaR values}\label{table:comp_Hana}
\centering
\begin{subtable}{.45\linewidth}
\centering
\caption{Bimodal case}
\resizebox{!}{3cm}{
\begin{tabular}{ccc}
\toprule[1.5pt]
Model & Mean & Variance\\
\midrule[1pt]
SP & -211.7763 & 66.5670   \\
MDRO-D & -200.0447 & 313.4354 \\
MDRO-T & -219.5837 &  -  \\
MGDRO-M1 & -214.4576 & 33.1766   \\
MGDRO-M2 & -214.5496 & 32.1666 \\
MGDRO-W1 & -211.5750 & 67.5989  \\
MGDRO-W2 & -211.5001 & 66.7781  \\
\bottomrule[1.5pt]
\end{tabular}
}
\end{subtable}
\begin{subtable}{.45\linewidth}
\centering
\caption{Trimodal case}
\resizebox{!}{3cm}{
\begin{tabular}{ccc}
\toprule[1.5pt]
Model & Mean & Variance\\
\midrule[1pt]
SP & -519.1817 &  118.2426  \\
MDRO-D & -338.9905 & 95658.8055 \\
MDRO-T & -526.9637 &  -  \\
MGDRO-M1 & -522.2796 &  57.0337  \\
MGDRO-M2 & -523.4279 & 61.6961 \\
MGDRO-W1 & -520.2344 & 113.1492  \\
MGDRO-W2 & -521.1914 & 105.9388  \\
\bottomrule[1.5pt]
\end{tabular}
}
\end{subtable} 
\end{table}  

The MDRO-D is inferior to MGDRO models, while MDRO-T provides the best out-of-sample CVaR value. The sample mean and covariance matrix are good estimation of the true mean and covariance matrix, while the estimations of the support sets under each state are relatively rough. It demonstrates that the performance of the MDRO model depends a lot on the estimations of the support sets under each state. The MDRO model is a great choice as long as the support sets are known, or estimated very well. When the support sets are unknown, the MGDRO models have more advantages, due to the globalization. The sufficiently large sample space contains the whole support set of the true probability distribution, so the MGDRO model is robust enough even if the core sets are not appropriate estimations of the true support sets under each state. Hence, the MGDRO framework is a great choice to handle the optimization problems with multimodal uncertain data if there is only a set of samples of uncertain data.

\section{Conclusions}\label{sec:cons}
In this study, we introduce a globalized distributionally robust optimization framework, efficiently handling uncertain problems with multimodal data. The sample space is sufficiently large to guarantee the support set of the true probability distribution is contained. The multiple core sets allow us to capture the distribution regions around each mode. The penalty coefficients $\theta_i,i=1,\ldots,m$ reflect the degree of preference for the core sets and control the degree of conservatism flexibly. The penalty item composed of the distance from the random vector to core sets and the penalty coefficients, weakens the impact of the regions outside of the core sets on the expectation of the objective function, thereby highlights the impact of the multimodality. Hence, the MGDRO framework utilizes the information of the sample data more efficiently, thereby reduces the degree of conservatism and preserves the robustness simultaneously. Under some assumptions, the MGDRO models with moment-based and metric-based ambiguity sets are both computionally tractable. We apply the MGDRO models to a multi-product newsvendor problem with multimodal demand distributions. The numerical results show that the MGDRO models outperform other models significantly, indicating the benefits of globalization and multiple core sets.

\clearpage
\bibliographystyle{plain}
\bibliography{ref}

\end{document}